\documentclass[11pt]{amsart}
\usepackage{amssymb,color,mathtools}
\usepackage{MnSymbol}

\numberwithin{equation}{section}

\newcommand{\sfrac}[2]{#1/#2}

\newcommand{\N}{\mathbb{N}}
\newcommand{\Z}{\mathbb{Z}}
\newcommand{\R}{\mathbb{R}}
\newcommand{\C}{\mathbb{C}}
\newcommand{\Q}{\mathbb{Q}}
\newcommand{\T}{\mathbb{T}}
\newcommand{\PP}{\mathbb{P}}
\newcommand{\Rplus}{\R^+}
\newcommand{\Zplus}{\Z^+}

\renewcommand{\Re}{\operatorname{Re}}

\newcommand{\sgn}{\operatorname{sgn}}
\newcommand{\dd}{\ignorespaces\,\mathrm{d}}
\newcommand{\expe}{\mathrm{e}}
\newcommand{\half}{\frac{1}{2}}
\newcommand{\origin}{o}
\newcommand{\pinfty}{+\infty}
\newcommand{\minfty}{-\infty}

\newcommand{\lpar}{\left(}
\newcommand{\rpar}{\right)}
\newcommand{\rabs}{\right\vert}
\newcommand{\labs}{\left\vert}
\newcommand{\lset}{\left\lbrace}
\newcommand{\rset}{\right\rbrace}
\newcommand{\lnorm}{\left\Vert}
\newcommand{\rnorm}{\right\Vert}

\let\oldepsilon\epsilon
\let\oldvarepsilon\varepsilon
\renewcommand{\epsilon}{\oldvarepsilon}
\renewcommand{\varepsilon}{\oldepsilon}
\let\oldphi\phi
\let\oldvarphi\varphi
\renewcommand{\phi}{\oldvarphi}
\renewcommand{\varphi}{\oldphi}

\newtheorem{theorem}{Theorem}[section]
\newtheorem{proposition}[theorem]{Proposition}
\newtheorem{corollary}[theorem]{Corollary}
\newtheorem{lemma}[theorem]{Lemma}

\theoremstyle{definition}

\theoremstyle{remark}
\newtheorem{remark}[theorem]{Remark}

\newcommand{\oldk}{h}
\newcommand{\oldh}{\tilde k}
\newcommand{\oldzeta}{r}
\newcommand{\oldazeta}{\tilde a_r}

\title[Estimates for matrix coefficients]{Estimates for matrix coefficients of representations }

\author[Bruno]{Tommaso Bruno}
\address{Dipartimento di Scienze Matematiche ``Giuseppe Luigi Lagrange'',
  Politecnico di Torino, Corso Duca degli Abruzzi 24, 10129 Torino,
  Italy}
\email{tommaso.bruno@polito.it}

\author[Cowling]{Michael G.~Cowling}
\address{School of Mathematics and Statistics,
  University of New South Wales, UNSW Sydney 2052,
 Australia}
\email{m.cowling@unsw.edu.au}

\author[Nicola]{Fabio Nicola}
\address{Dipartimento di Scienze Matematiche ``Giuseppe Luigi Lagrange'',
  Politecnico di Torino, Corso Duca degli Abruzzi 24, 10129 Torino,
  Italy}
\email{fabio.nicola@polito.it}

\author[Tabacco]{Anita Tabacco}
\address{Dipartimento di Scienze Matematiche ``Giuseppe Luigi Lagrange'',
  Politecnico di Torino, Corso Duca degli Abruzzi 24, 10129 Torino,
  Italy}
\email{anita.tabacco@polito.it}

\thanks{T.~B., F.~N.~and A.~T.~were partially supported by PRIN 2015--2018 ``Variet\`a reali e complesse: geometria, topologia e analisi armonica'' (2015A35N9B\_001).
M.~C.\ was supported by the award of a Fubini Visiting Professorship at the Department of Mathematical Sciences of the Politecnico di Torino, and by the Australian Research Council (DP170103025).
We acknowledge that the present research has been partially supported by
MIUR grant Dipartimenti di Eccellenza 2018--2022 (E11G18000350001).
}


\dedicatory{In memory of Elias M. Stein}

\begin{document}

\begin{abstract}
Estimates for matrix coefficients of unitary representations of semisimple Lie groups have been studied for a long time, starting with the seminal work by Bargmann, by Ehrenpreis and Mautner, and by Kunze and Stein.
Two types of estimates have been established: on the one hand, $L^p$ estimates, which are a dual formulation of the Kunze--Stein phenomenon, and which hold for all matrix coefficients, and on the other pointwise estimates related to asymptotic expansions at infinity, which are more precise but only hold for a restricted class of matrix coefficients.
In this paper we prove a new type of estimate for the irreducibile unitary representations  of $\mathrm{SL}(2,\R)$ and for the so-called metaplectic representation, which we believe has the best features of, and implies, both forms of estimate described above.
As an application outside representation theory, we prove a new $L^2$  estimate of dispersive type for the free Schr\"odinger equation in $\R^n$.
\end{abstract}

\maketitle

\section{Introduction}
Let $\pi$ be a strongly continuous unitary representation of a locally compact group $G$ on a Hilbert space $\mathcal{H}_\pi$.
A matrix coefficient of $\pi$ is a function on $G$ of the form $x \mapsto \langle \pi(x) \xi, \eta \rangle$, where $\xi, \eta \in \mathcal{H}_\pi$.
These matrix coefficients encode the properties of $\pi$.

We consider the particular case where $G \coloneqq \mathrm{SL}(2,\R)$.
This group has two special subgroups $A$ and $K$: the former consists of diagonal matrices and the latter of rotation matrices; more precisely, we set
\begin{equation}\label{def:K-and-A}
a_r \coloneqq \begin{pmatrix}
\expe^{-r} & 0 \\
0 & \expe^{r}
\end{pmatrix}
\qquad\text{and}\qquad
k_{\theta} \coloneqq \begin{pmatrix}
\cos\theta & \sin\theta \\
-\sin\theta & \cos\theta
\end{pmatrix}
\end{equation}
for all $r, \theta \in \R$.
Note that every element $x$ of $G$ admits a Cartan decomposition, that is, we may write  $x = k_{\theta_1} a_r k_{\theta_2}$ where ${\theta_1}, {\theta_2} \in \R$  and $r \in \Rplus\cup\{0\}$.

We are going to prove estimates for matrix coefficients of irreducible representations $\pi$ of the form
\begin{equation}\label{eq:basic-estimate}
\lpar \int_{-\pi}^{\pi} \int_{-\pi}^{\pi} \labs  \langle \pi(k_{\theta_1} a_r k_{\theta_2}) \xi, \eta \rangle \rabs ^2 \dd{\theta_1} \dd{\theta_2} \rpar^{1/2}
\leq C  \lnorm  \xi   \rnorm_{\mathcal{H}_\pi}  \lnorm  \eta  \rnorm_{\mathcal{H}_\pi}  \exp (- \lambda r)
\end{equation}
for all $\xi, \eta \in \mathcal{H}_{\pi}$  and all $r \in \Rplus$; the real parameter $\lambda$ depends on $\pi$.
These generalise and extend various estimates that have been proved over the years, which we now describe in more detail.

The unitary representations and harmonic analysis of $G$ have been studied for many years, starting with the fundamental paper of Bargmann \cite{Bargmann}, which identified three families of irreducible unitary representations of $G$, namely, the principal series, the discrete series, the complementary series, and one additional representation, the trivial representation.
We describe these in detail below.
Amongst many other things, Bargmann found explicit formulae in terms of special functions for the ``generalised spherical functions'', that is, the matrix coefficients $\langle \pi(\cdot) \xi, \eta \rangle$ of these representations $\pi$ when the vectors $\xi$ and $\eta$ are particular normalised vectors that transform by scalars under the action of $K$, that is, $\pi(k_{\theta})\xi = \expe^{in{\theta}}\xi$ and $\pi(k_{\theta})\eta = \expe^{im{\theta}}\eta$.
These led to asymptotic formulae of the form
\[
\langle \pi(k_{\theta_1} a_r k_{\theta_2}) \xi, \eta \rangle
\sim C_1(\pi, m,n) \exp (-\lambda_1 r) + C_2(\pi, m,n) \exp (-\lambda_2 r)
\]
as $r \to \pinfty$;  the (possibly complex) parameters $\lambda_1$ and $\lambda_2$ depend on which representation is under consideration, and in some cases only one term is present.
These formulae were instrumental in his proof of the Plancherel formula, which involves the representations of the principal and discrete series only.

Bargmann observed that all the generalised spherical functions associated to the discrete series belong to $L^2(G)$; those associated to the principal series belong to $L^{2+}(G)$, by which we mean that they belong to $L^{2+\epsilon}(G)$ for all $\epsilon \in \Rplus$; those associated to the complementary series belong to $L^{p+}(G)$ for some $p$ depending only on the representation.
The matrix coefficients of the trivial representation are all constants that do not decay at infinity at all.
The estimates that follow from his analysis do not seem to be uniform in $r$ when we consider different ``$K$-types'' $m$ and $n$ and different representations.

Another great leap forward was the work of Kunze and Stein \cite{Kunze-Stein}.
They showed that every matrix coefficient of every representation involved in the Plancherel formula, and hence every matrix coefficient of the regular representation, lies in $L^{2+}(G)$.
This is a dual formulation of a convolution estimate $ \lnorm f*g  \rnorm_2 \leq C(p) \lnorm f \rnorm_p \lnorm g \rnorm_2$ for all $p \in [1,2)$, now known as the Kunze--Stein phenomenon; see \cite{Cowling-Annals} for more details of this equivalence.
Kunze and Stein also established $L^{p+}$ estimates for the complementary series.
A typical estimate is of the form
\[
\lnorm  \langle \pi(\cdot) \xi, \eta \rangle   \rnorm_{q} \leq C(\pi, q) \lnorm  \xi   \rnorm_{\mathcal{H}_\pi}  \lnorm  \eta  \rnorm_{\mathcal{H}_\pi}
\qquad\forall \xi, \eta \in \mathcal{H}_{\pi}\quad\forall q \in (p,\pinfty).
\]

As time went on, other forms of decay estimates were established for $\mathrm{SL}(2,\R)$ and for more general groups.
We mention in particular the pointwise estimates of Howe and Tan \cite{Howe-Tan} for $\mathrm{SL}(2,\R)$ and the estimates of Howe \cite{Howe-CIME} for more general groups, as well as the $L^p$ estimates of the second author \cite{Cowling-Nancy}.
The asymptotic formulae found by Bargmann (see also Ehrenpreis and Mautner \cite{Ehrenpreis-Mautner}) have been generalised to general semisimple Lie groups to give asymptotic expansions of $K$-finite matrix coefficients of irreducible representations by Harish-Chandra.
See Warner \cite{Warner} or Casselman and Mili\v{c}i\'{c} \cite{Casselman-Milicic} for a comprehensive exposition; see also Knapp \cite{knapp} and Wallach \cite{wallach}.
These pointwise estimates are very precise ``at infinity'', but they only hold for some matrix coefficients and it is hard to see the sort of uniform behaviour that the Kunze--Stein phenomenon tells us must occur.
On the other hand, $L^{p+}$ estimates hold for all matrix coefficients of a given irreducible unitary representation, but as students of Lebesgue integration know, the fact that a function lies in some $L^q$-space does not mean much.  Despite this, $L^{p+}$ estimates have also found applications in representation theory and in related areas; see, for example, \cite{Nevo} and \cite{Benoist-Kobayashi}.

The aim of this paper is to present a new form of estimates for matrix coefficients, which we believe has the best features of both of the forms of estimate above.
These estimates were inspired by similar estimates for the free group due to Haagerup \cite{Haagerup-free} that have been extended to estimates for representations of groups of isometries of trees, and in particular, therefore, to groups such as $\mathrm{SL}(2, \Q_p)$; see \cite{CMS} for more information.
We treat only $\mathrm{SL}(2,\R)$ and some special representations of the metaplectic group.
We prove estimates for all matrix coefficients that reduce to sharp forms of the pointwise estimates of Howe \cite{Howe-CIME} for $K$-finite matrix coefficients and imply similar estimates to the $L^p$ estimates of Kunze--Stein \cite{Kunze-Stein} for $\mathrm{SL}(2,\R)$.
It would be nice to extend these to more general semisimple Lie groups in the future, and we envisage applications in representation theory for these.

As an application outside representation theory, we present some multi-dimensional dispersive estimates that go beyond the current fascination with dispersive estimates for the Schr\"odinger equation (see Tao \cite{tao} and also \cite{CNT} for a representation theoretic perspective).\par\medskip
Here is a plan of the rest of this work.
In Section 2, we describe the principal and complementary series, and in Section 3 the discrete series.
In Section 4 we prove estimates of the form \eqref{eq:basic-estimate}.
For the principal and complementary series, we use the approach of  Astengo, Cowling and Di Blasio~\cite{ACD1}; for the discrete series, we use that of Bargmann~\cite{Bargmann}.
In Section 5, we consider the matrix coefficients of the metaplectic representation and we provide the above mentioned application to the Schr\"odinger equation.

We write $A \lesssim B$ to indicate that there is a constant $C$ such that $A \leq CB$.
The implied constants $C$ do not depend on explicitly quantified parameters.
All ``constants'' are positive.

\section{The principal and complementary series}

We write elements of $\R^2$ as row vectors, so that $G$ acts on $\R^2$ by right multiplication, and $\origin$ for the origin.
For $\zeta \in \C$ and $\epsilon\in \lset 0,1\rset $, we define $\mathcal{V}_{\zeta, \epsilon}$ to be the space of smooth functions on $\R^2 \setminus \lset  \origin \rset$ such that
\[
f(\delta v)
= \sgn(\delta)^\epsilon \labs \delta \rabs^{2\zeta -1} f(v) \qquad\forall v\in \R^2 \setminus \lset  \origin\rset  \quad\forall\delta\in \R \setminus \lset 0\rset,
\]
and the representation $\pi_{\zeta, \epsilon}$ of $G$ on $\mathcal{V}_{\zeta, \epsilon}$ by
\[
\pi_{\zeta, \epsilon}(x) f(v) \coloneqq f(vx)
\qquad\forall v\in \R^2 \setminus \lset  \origin\rset  \quad\forall x\in G.
\]
Observe that the functions in $\mathcal{V}_{\zeta, \epsilon}$ are determined by their values on the unit circle.
Hence the space $\mathcal{V}_{\zeta, \epsilon}$ is spanned topologically by functions whose restriction to the unit circle is a complex exponential.
For every $\mu \in \Z$, we denote by $f_{\zeta, \mu}$ the function in $\mathcal{V}_{\zeta, \epsilon}$ such that
\[
f_{\zeta, \mu}(\cos\theta, \sin \theta)
= \frac{1}{\pi^{1/2}} \expe^{i\mu \theta}
\qquad\forall \theta \in \R.
\]
This forces $\mu-\epsilon$ to be even.

We may define the pairing (see~\cite[Lemma 3.2]{ACD1})
\begin{equation}\label{eq:pairing}
\begin{aligned}
(f,g)
&\coloneqq  \frac{1}{2}   \int_{-\pi}^{\pi} f(\cos\theta, \sin\theta) g(\cos\theta, \sin\theta)  \dd \theta\\
& =  \int_{-\pi/2}^{\pi/2} f(\cos\theta, \sin\theta) g(\cos\theta, \sin\theta) \dd \theta\\
& = \int_\R f(1,t)g(1,t) \dd t.
\end{aligned}
\end{equation}
We recall that if either  $ \Re \zeta =0$ (the \emph{principal series}), or
$\zeta \in \pm(0,\frac{1}{2}) $ and $\epsilon =0$ (the \emph{complementary series}),
then $\mathcal{V}_{\zeta,\epsilon}$ may be endowed with an inner product whose completion is a Hilbert space on which the representation $\pi_{\zeta,\epsilon}$ acts unitarily and, except when $\zeta=0$ and $\epsilon = 1$, is irreducible.

\subsection{The principal series}
If $\zeta = \frac{1}{2} is$, where $s\in \R$ (the so-called \emph{principal series}), the completion of the space $\mathcal{V}_{\frac{1}{2}is,\epsilon}$ with respect to the inner product
\begin{align*}
\langle f , g\rangle  & \coloneqq (f,\bar{g})
\end{align*}
is a Hilbert space $\mathcal{H}_{\frac{1}{2}is, \epsilon} $ on which $\pi_{\frac{1}{2}is,\epsilon}$ acts unitarily.
Moreover, the functions $f_{\frac{1}{2}is,\mu}$ form an orthonormal basis,
and each transforms under the action of $K$ by a complex exponential (recall the definitions \eqref{def:K-and-A} of $A$ and $K$).

\begin{theorem}\label{teo_principal}
For all $s\in \R\setminus \lset 0\rset $, all $\epsilon \in \lset 0,1\rset$ and all $\mu,\nu \in 2\Z+\epsilon$,
\[
\labs \langle \pi_{\frac{1}{2}is, \epsilon} (a_r) f_{\frac{1}{2}is,\mu}, f_{\frac{1}{2}is,\nu} \rangle  \rabs\lesssim \frac{|s|+1}{|s|} \expe^{-r}
\qquad\forall r\in \Rplus.
\]
\end{theorem}

\begin{proof}
The result is trivially true if $r \in (0,1)$, so we may and shall assume that $r \in [1, \pinfty)$.
We abbreviate $f_{\frac{1}{2}is,\mu}$ to $f_\mu$.

From the properties of $f_{\mu}$,
\begin{equation}\label{arctanprop}
\begin{aligned}
\langle \pi_{\frac{1}{2}is, \epsilon} (x) f_{\mu}, f_{\nu} \rangle
&= \int_\R f_{\mu} (\expe^{-r},\expe^r t) {f}_{-\nu}(1,t)  \dd t \\
& = \expe^{-r} \int_\R  f_\mu (\expe^{-r}, t) {f}_{-\nu}(1,\expe^{-r}t)  \dd t \\
& = \frac{1}{\pi}\expe^{-r} \int_\R (\expe^{-2r} +t^2)^{(is-1)/2} \expe^{i\mu \arctan (\expe^r t)} \\
&\qquad\qquad\qquad \times (1+ \expe^{-2r} t^2)^{-(is+1)/2} \expe^{-i\nu \arctan (\expe^{-r} t)} \dd t \\
&= \frac{1}{\pi}\expe^{-r} \int_{0}^{\pinfty} \dots \dd t
+ \frac{1}{\pi}\expe^{-r} \int_{\minfty}^{0} \dots \dd t .
\end{aligned}
\end{equation}
The second integral on the last line is the same as the first with the signs of $\mu$ and $\nu$ changed;  since we are looking for estimates that are uniform in $\mu$ and $\nu$ it is enough to consider only the first integral, and bound it by a multiple of $(|s|+1)/|s|$; we change variables, setting $t \coloneqq e^y$, and rewrite the integral in the form
\begin{equation*}
\begin{aligned}
&\int_\R (\expe^{-2r} + \expe^{2y})^{(is-1)/2} \expe^{i\mu \arctan (\expe^{y+r} )} (1 + \expe^{-2r} \expe^{2y})^{-(is+1)/2} \expe^{-i\nu \arctan (\expe^{y-r})} \expe^y \dd y \\
&\qquad = \int_\R \alpha_r(y) \exp(i \phi_{r,s,\mu,\nu}(y)) \dd y,
\end{aligned}
\end{equation*}
where $\alpha_r(y)$ is equal to $[(1+ \expe^{-2y-2r})(1 + \expe^{2y-2r})]^{-1/2}$, which lies in $(0,1)$, and $\phi_{r,s,\mu,\nu}(y)$ is equal to
\[
\begin{aligned}
\frac{s}{2} [2y + \log( 1 + \expe^{-2y-2r}) -\log(1 + \expe^{2y-2r})] 
+ \mu \arctan(\expe^{y+r}) - \nu \arctan(\expe^{y-r}) .
\end{aligned}
\]
The right hand side of the last integral changes to its complex conjugate when $s$, $\mu$ and $\nu$ all change signs, so we may suppose that $s \in \Rplus$.

For brevity, we write $\tilde r \coloneqq r-1$.
Now
\begin{align*}
 \int_{\minfty}^{-\tilde r}  \alpha_r(y)  \dd y
&\leq \int_{\minfty}^{-\tilde r}  \expe^{y+r}  \dd y = \expe ,\\
 \int_{\tilde r}^{{\pinfty}} \alpha_r(y)  \dd y
&\leq \int_{\tilde r}^{\pinfty}  \expe^{-y-r}  \dd y = \expe,\\
\noalign{\noindent{and}}
\int_{-\tilde r}^{\tilde r}  1 - \alpha_r(y)  \dd y
&\leq \frac{1}{2} \int_{-\tilde r}^{\tilde r}  \expe^{-2y-2r} + \expe^{2y-2r} + \expe^{-4r} \dd y \\
&= \expe^{-2r} \sinh(2\tilde r) + \expe^{-4r}\tilde r
\leq \frac{1}{2\expe^2}  + \frac{1}{4\expe^5} \,,
\end{align*}
since $1 - [1 + z]^{-1/2} \leq \half z$ when $z \in \Rplus$.
Because
\begin{align*}
&\labs \int_\R \alpha_r(y) \exp(i \phi_{r,s,\mu,\nu}(y)) \dd y \rabs \\
&\qquad\leq \int_{\minfty}^{-\tilde r} \labs \alpha_r(y) \rabs \dd y
 + \int_{\tilde r}^{\pinfty} \labs \alpha_r(y) \rabs \dd y
 + \int_{-\tilde r}^{\tilde r} \labs \alpha_r(y) -1 \rabs \dd y \\
&\qquad\qquad\qquad+ \labs \int_{-\tilde r}^{\tilde r} \exp(i \phi_{r,s,\mu,\nu}(y)) \dd y \rabs,
\end{align*}
it will suffice to show that
\begin{equation*}
I \coloneqq \labs \int_{-\tilde r}^{\tilde r} \exp(i \phi_{r,s,\mu,\nu}(y)) \dd y \rabs
\lesssim \frac{s+1}{s} \,.
\end{equation*}
Let $N(\mu,\nu, r, s)$ be the number of zeros of $\phi_{r,s,\mu,\nu}''$ in $[-\tilde r, \tilde r]$.
Then van der Corput's Lemma applied to the integral over $\{ y \in [-\tilde r, \tilde r] :  \labs \phi_{r,s,\mu,\nu}'(y)\rabs \geq s/2\}$ and the trivial estimate applied to the complementary integral show that
\[
I \lesssim \frac{N(\mu,\nu, r, s)}{s} + \labs\lset  y \in [-\tilde r,\tilde r] : \labs \phi_{r,s,\mu,\nu}'(y) \rabs  < s/2\rset\rabs.
\]
An easy computation shows that
\[
\phi_{r,s,\mu,\nu}'(y)= s \left[ 1-  \frac{ \expe^{-2(y+r)}}{1+\expe^{-2(y+r)}} -  \frac{ \expe^{2(y-r)}}{1+\expe^{2(y-r)}}\right]+ \frac{\mu}{2\cosh(y+r)} -  \frac{\nu}{2\cosh(y-r)} \,.
\]
Now $\phi_{r,s,\mu,\nu}'$ is a rational function of $\expe^y$ with coefficients that depend on $r$, $s$, $\mu$ and $\nu$, so the same holds for $\phi_{r,s,\mu,\nu}''$; therefore the number of zeros of $\phi_{r,s,\mu,\nu}''$ is uniformly bounded with respect to these parameters.
Thus $N(\mu,\nu, r, s) \leq N$ for a universal constant $N$.
It remains only to show that
\begin{equation}\label{gmunursinterval}
\labs\lset y \in [-\tilde r,\tilde r] : \labs \phi_{r,s,\mu,\nu}'(y) \rabs < s/2 \rset\rabs
\lesssim 1.
\end{equation}

Observe that
\[
1- \frac{\expe^{-2(y+r)}}{1+\expe^{-2(y+r)}}
- \frac{\expe^{2(y-r)}}{1+\expe^{2(y-r)}}
= \frac{1 - \expe^{-4r}}{1+ \expe^{-2y-2r} + \expe^{2y-2r} +\expe^{-4r}}
\,,
\]
and to minimise this last expression for $y$ in $[-\tilde r,\tilde r]$, we take $y = \tilde r$ and obtain
\[
1- \frac{\expe^{-2(y+r)}}{1+\expe^{-2(y+r)}}
- \frac{\expe^{2(y-r)}}{1+\expe^{2(y-r)}}
\geq \frac{1 - \expe^{-4r}}{1+ \expe^{2-4r} + \expe^{-2} +\expe^{-4r}}
\,.
\]
As a function of $r$, the numerator of this function is increasing while the denominator is decreasing, so we conclude that
\[
1 > 1- \frac{\expe^{-2(y+r)}}{1+\expe^{-2(y+r)}}
- \frac{\expe^{2(y-r)}}{1+\expe^{2(y-r)}}
\geq \frac{1 - \expe^{-4}}{1+ \expe^{-2} + \expe^{-2} +\expe^{-4}}
= \tanh(1) > \frac{3}{4}
\]
for all $r \in [1, \pinfty)$ and all $y \in [-\tilde r, \tilde r]$.
Define
\[
\psi_{r,\mu,\nu}(y)
\coloneqq \frac{\nu}{2\cosh(y-r)} - \frac{\mu}{2\cosh(y+r)} ;
\]
then $\labs \phi_{r,s,\mu,\nu}'(y) \rabs<s/2$ implies that $\psi_{r,\mu,\nu}(y) \in [s/4, 3s/2 ]$.
We consider two cases.

If $\mu\nu\geq0$, then $\psi_{r,\mu,\nu}$ is monotone (increasing if either both $\mu>0$ and $\nu>0$ or both $\mu=0$ and $\nu>0$; decreasing if either both $\mu<0$ and $\nu<0$ or both $\mu<0$ and $\nu=0$).
Moreover, $\labs \lset y \in [-\tilde r,\tilde r] : \psi_{r,\mu,\nu}(y) \in [s/4, 3s/2]\rset   \rabs$ is equal to
\begin{align*}
\labs \lset  y \in [-\tilde r,\tilde r] :  \psi_{r,\mu,\nu}(y)>0,\;  \log(\psi_{r,\mu,\nu}(y))\in [\log(s/4), \log(3s/2)]\rset  \rabs.
\end{align*}
To prove that this quantity is bounded by a constant independent of $r$, $s$, $\mu$, and $\nu$, it suffices to show that the derivative of $\log(\psi_{r,\mu,\nu}(\cdot))$ is bounded away from $0$, uniformly in $r$, $\mu$, and $\nu$, that is,
\[
\labs \psi_{r,\mu,\nu}'(y) \rabs \gtrsim \psi_{r,\mu,\nu}(y).
\]
There are various cases to consider; for instance, if $\mu>0$ and $\nu>0$, then
\begin{align*}
\labs \psi_{r,\mu,\nu}'(y) \rabs
&= \psi_{r,\mu,\nu}'(y)
= \frac{\nu \sinh(r-y)}{2\cosh^2(r-y)} + \frac{\mu \sinh(r+y)}{2\cosh^2(r+y)} \\
& = \frac{\nu}{2\cosh(r-y)} \tanh(r-y) + \frac{\mu}{2\cosh(r+y)} \tanh(r+y)\\
& \geq \tanh(1)\frac{\nu}{2\cosh(r-y)}
 \geq \tanh(1) \psi_{r,\mu,\nu}(y).
\end{align*}
The other cases that arise when $\mu\nu \geq 0$ may be treated analogously.

If $\mu\nu<0$, then the set $\lset y \in [-\tilde r,\tilde r] : \psi_{r,\mu,\nu}(y) \in [s/4,3s/2]\rset $ is empty unless $\mu<0$ and $\nu>0$.
In this case, $\psi_{r,\mu,\nu}>0$.
Moreover, for $y$ in $[-\tilde r,\tilde r]$, there are uniform estimates
\[
\half \expe^{r+y} \leq \cosh(r+y) \leq \expe^{r+y}
\qquad\text{and}\qquad
\half \expe^{r-y} \leq \cosh(r-y) \leq \expe^{r-y}.
\]
Hence if $\sfrac{s}{4} \leq  \psi_{r,\mu,\nu}(y) \leq \sfrac{3s}{2}$, then $\sfrac {s \expe^{r}}{4} \leq \nu \expe^{y}-\mu \expe^{-y} \leq 3 s \expe^r$, and so
\[
\frac {s \expe^{r}}{8\labs\mu\nu\rabs^{1/2}} \leq  \frac{\expe^{y + z} + \expe^{-y-z}}{2} \leq \frac{3 s \expe^r}{2\labs\mu\nu\rabs^{1/2}} \,,
\]
where $z \coloneqq \half \log(\sfrac{\nu}{|\mu|})>0$.

We claim that
\[
\labs\lset y \in \R : \cosh(y+z) \in [a,b] \rset\rabs \leq 2 \log (2b/a)
\]
whenever $0 < a < b < \pinfty$ and $z \in \R$.
From this claim, it follows  immediately that
\[
\labs\lset y \in [-\tilde r, \tilde r] : \psi_{r,\mu,\nu}(y) \in [s/4,3s/2] \rset\rabs
\leq 2 \log (24),
\]
as required.
So we need to establish our claim.

Evidently, the truth (or otherwise) of the claim does not depend on $z$, so we shall assume that $z = 0$.
By symmetry, the size of the required set is twice $\labs\lset y \in \Rplus : \cosh(y) \in [a,b] \rset\rabs$.
But
\[
\labs\lset y \in \Rplus : a \leq \cosh(y) \leq b \rset\rabs
\leq
\labs\lset y \in \Rplus : a \leq \expe^y \leq 2b \rset\rabs
= \log(2b) - \log(a),
\]
and the claim follows.

The proof of~\eqref{gmunursinterval} and thus of the theorem is complete.
\end{proof}

\subsection{The complementary series}
Let $\Re \zeta \in (-\frac{1}{2}, 0)$, and define the intertwining operator $J_{\zeta, \epsilon}$ by setting $J_{\zeta, \epsilon} f (\cos \phi, \sin\phi)$ equal to
\[
\frac{1}{2c(\zeta, \epsilon)} \int_{-\pi}^\pi f(\cos\theta,\sin\theta) \sgn^\epsilon (\sin(\phi -\theta))\labs \sin (\theta-\phi) \rabs^{-1-2\zeta} \dd \theta
\]
where $c(\zeta,\epsilon)\coloneqq i^\epsilon \pi^{1/2}2^{-2\zeta} \frac{\Gamma(\sfrac{\epsilon}{2}-\zeta)}{\Gamma(\sfrac{(1+\epsilon)}{2}+\zeta)}$. It is known (see, for instance, \cite[Lemma 3.6]{ACD1}) that, for these $\zeta$,
\begin{itemize}
\item $J_{\zeta,\epsilon}$ maps $\mathcal{V}_{\zeta, \epsilon}$ bijectively and bicontinuously onto $\mathcal{V}_{-\zeta,\epsilon}$;
\item $J_{\zeta,\epsilon} f_{\zeta,\mu} = d(\zeta,\epsilon, \mu) f_{-\zeta,\mu}$, where $d(\zeta, \epsilon, \mu)= 2^{2\zeta} \frac{\Gamma(\sfrac{1}{2} +\zeta +\sfrac{\mu}{2})}{\Gamma(\sfrac{1}{2}-\zeta+\sfrac{\mu}{2})}$;
\item $J_{\zeta,\epsilon}\pi_{\zeta,\epsilon}(x)= \pi_{-\zeta,\epsilon}(x)J_{\zeta,\epsilon}$ for all $x\in G$;
\end{itemize}
further, the map $\zeta \mapsto J_{\zeta,\epsilon}$ extends analytically to $\lset \zeta \in \C : \Re\zeta \in (-\frac{1}{2},\frac{1}{2}) \rset $ and the three properties above continue to hold in this region.
Therefore if $\Re\zeta \in \pm (0,\frac{1}{2})$, for all $f,g\in \mathcal{V}_{\zeta,\epsilon}$ the pairing (as in \eqref{eq:pairing})
\[
(J_{\zeta,\epsilon} f, g) = \frac{1}{2}   \int_{-\pi}^{\pi} J_{\zeta,\epsilon} f(\cos\phi, \sin\phi) g(\cos\phi, \sin\phi) \dd \phi
\]
is well defined.

Now take $\zeta = \lambda \in \pm (0,\frac{1}{2})$ and $\epsilon =0$, and write $\pi_\lambda$, $\mathcal{V}_\lambda$ and $J_{\lambda}$ for $\pi_{\lambda,0}$, $\mathcal{V}_{\lambda, 0}$ and $J_{\lambda,0}$ respectively. This corresponds to the so-called \emph{complementary series}.
In this case, one may define on $\mathcal{V}_{\lambda}$ an inner product, written $\llangle\cdot, \cdot\rrangle $ to distinguish it from the previous one:
\begin{align*}
\llangle f , g\rrangle & \coloneqq (J_{\lambda}f,\bar{g}),
\end{align*}
and the completion of $\mathcal{V}_{\lambda}$ with respect to this inner product is a Hilbert space $\mathcal{H}_{\lambda} $ on which $\pi_{\lambda}$ acts unitarily (see, for example, \cite[Lemma 3.4]{ACD1}).
We write $d(\lambda,\mu)$ for $d(\lambda,0,\mu)$, and then
\begin{align*}
\lnorm f_{\lambda,\mu} \rnorm_{\mathcal{H}_\lambda}^2
&= \bigl\llangle f_{\lambda,\mu}, f_{\lambda,\mu}\bigr\rrangle
 = \lpar J_{\lambda} f_{\lambda,\mu}, \bar{f}_{\lambda,\mu} \rpar
 = d(\lambda,\mu) \lpar f_{-\lambda,\mu}, \bar{f}_{\lambda,\mu} \rpar \\
&=  d(\lambda,\mu) \lpar {f}_{-\lambda,\mu},  {f}_{\lambda,-\mu} \rpar
 =  d(\lambda,\mu).
\end{align*}
With respect to the inner product $\llangle\cdot,\cdot\rrangle$, the functions $f_{\lambda,\mu}$ are not normalised.
By the previous computation, it is clear that $d(\lambda, \mu)> 0$.
This may be seen explicitly using the recurrence formula for the gamma function, which implies that
\[
\frac{\Gamma(\sfrac{1}{2}+\lambda+\sfrac{\mu}{2})}{\Gamma(\sfrac{1}{2}-\lambda+\sfrac{\mu}{2})} = (-1)^\mu \frac{\Gamma(\sfrac{1}{2}+\lambda-\sfrac{\mu}{2})}{\Gamma(\sfrac{1}{2}-\lambda-\sfrac{\mu}{2})}
\]
where, in our case, $\mu$ is even since $\epsilon=0$ (see, for instance, \cite[eq.\ (2.8)]{ACD1}).
Hence we may define
\[
g_{\lambda,\mu}
\coloneqq \frac{1}{\lpar d(\lambda,\mu) \rpar^{1/2}} f_{\lambda,\mu}.
\]

We shall prove the following.
\begin{theorem}\label{teo_complementary}
For all $\lambda \in \pm (0,\frac{1}{2})$, all $r\in \Rplus$ and all $\mu,\nu\in 2\Z$,
\[
\labs\bigl\llangle \pi_{\lambda} (a_r)g_{\lambda,\mu}, g_{\lambda,\nu} \bigr\rrangle\rabs \lesssim \frac{1}{|\lambda|} \expe^{-r(1-2|\lambda|)}.
\]
\end{theorem}
\begin{proof} Again, we may and shall assume that $r \in [1, \pinfty)$.
Observe first that for all $x\in G$,
\begin{align*}
\bigl\llangle \pi_{\lambda} (x)g_{\lambda,\mu}, g_{\lambda,\nu}\bigr\rrangle
&= \lpar J_{\lambda}\pi_{\lambda}(x)g_{\lambda,\mu}, \bar{g}_{\lambda,\nu} \rpar\\
& = \frac{1}{\lpar d(\lambda,\mu)d(\lambda,\nu) \rpar^{1/2}} \lpar J_{\lambda}\pi_{\lambda}(x)f_{\lambda,\mu}, \bar{f}_{\lambda,\nu} \rpar\\
& = \frac{1}{\lpar d(\lambda,\mu)d(\lambda,\nu) \rpar^{1/2}} \lpar \pi_{-\lambda}(x)J_{\lambda} {f}_{\lambda,\mu}, {f}_{\lambda,-\nu}\rpar \\
& = \frac{{d(\lambda,\mu)}}{\lpar d(\lambda,\mu)d(\lambda,\nu) \rpar^{1/2}} \lpar\pi_{-\lambda}(x) f_{-\lambda,\mu}, {f}_{\lambda,-\nu} \rpar\\
& = \lpar \frac{d(\lambda,\mu) }{d(\lambda,\nu)} \rpar^{1/2}
\int_{\R}{ \pi_{-\lambda}(x) f_{-\lambda,\mu}(1,t)} {f}_{\lambda,-\nu}(1,t) \dd t.
\end{align*}
If $x=a_{\pm r}$, then, proceeding as in~\eqref{arctanprop}, we see that
\begin{align*}
&\llangle \pi_{\lambda} (a_{\pm r})g_{\lambda,\mu}, g_{\lambda,\nu}\rrangle  \\
&\qquad= \lpar \frac{d(\lambda,\mu) }{d(\lambda,\nu)} \rpar^{1/2} \int_\R f_{-\lambda,\mu} (\expe^{ \mp r},\expe^{\pm r} t) {f}_{\lambda,-\nu}(1,t)  \dd t\\
&\qquad= \lpar \frac{d(\lambda,\mu) }{d(\lambda,\nu)} \rpar^{1/2} \frac{\expe^{\mp r}}{\pi}  \int_\R (\expe^{\mp 2r} +t^2)^{-\frac{2\lambda+1}{2}} \expe^{i\mu \arctan (\expe^{\pm r} t)} \\
&\qquad\qquad\qquad\qquad\qquad\qquad\qquad \times (1+ \expe^{\mp 2r} t^2)^{\frac{2\lambda-1}{2}} \expe^{-i\nu \arctan (\expe^{\mp r} t)} \dd t\\
&\qquad\eqqcolon \lpar \frac{d(\lambda,\mu) }{d(\lambda,\nu)} \rpar^{1/2} A_{\lambda,\mu,\nu}(\pm r).
\end{align*}
We claim that
\begin{equation}\label{claimcomplementary}
 \labs A_{\lambda,\mu,\nu}(\pm r) \rabs \lesssim \frac{1}{|\lambda|} \expe^{-r(1-2|\lambda|)}.
\end{equation}
Assuming the claim, we prove the theorem. Since
\[
\llangle \pi_{\lambda} (a_r)g_{\lambda,\mu}, g_{\lambda,\nu}\rrangle = \llangle g_{\lambda,\mu}, \pi_{\lambda} (a_{-r})g_{\lambda,\nu}\rrangle  = \overline{\llangle \pi_{\lambda} (a_{-r})g_{\lambda,\nu}, g_{\lambda,\mu}\rrangle  }
\]
and
\[
\llangle \pi_{\lambda} (a_{\pm r})g_{\lambda,\mu}, g_{\lambda,\nu}\rrangle
= \lpar \frac{d(\lambda,\mu) }{d(\lambda,\nu)} \rpar^{1/2}  A_{\lambda, \mu,\nu}(\pm r),
\]
we have, by the claim,
\begin{align*}
\labs \llangle \pi_{\lambda} (a_r)g_{\lambda,\mu}, g_{\lambda,\nu}\rrangle\rabs
& = \labs \llangle \pi_{\lambda} (a_r)g_{\lambda,\mu}, g_{\lambda,\nu}\rrangle \llangle \pi_{\lambda} (a_{-r})g_{\lambda,\nu}, g_{\lambda,\mu}\rrangle \rabs^{1/2} \\
& = \lpar A_{\lambda, \mu,\nu}(r) A_{\lambda, \nu, \mu}( -r) \rpar^{1/2}
  \lesssim \frac{1}{|\lambda|} \expe^{-r(1-2|\lambda|)}.
\end{align*}
Therefore it remains to prove the claim.

To do this, we first observe that it is enough to prove~\eqref{claimcomplementary} for $A_{\lambda, \mu,\nu}(r)$, since
\[
A_{\lambda, \mu,\nu}(-r) = A_{-\lambda, -\nu,-\mu}(r).
\]
Thus, consider $A_{\lambda, \mu,\nu}(r)$.
First,
\begin{align*}
\pi \labs A_{\lambda, \mu,\nu}(r)  \rabs &\leq \expe^{-r} \int_\R (\expe^{-2r} +t^2)^{-\sfrac{(2\lambda+1)}{2}} (1+ \expe^{-2r} t^2)^{\sfrac{(2\lambda-1)}{2}}  \dd t \\
& = 2  \expe^{-r} \int_0^{\pinfty} (\expe^{-2r} +t^2)^{-\sfrac{(2\lambda+1)}{2}} (1+ \expe^{-2r} t^2)^{\sfrac{(2\lambda-1)}{2}}  \dd t\\
& = 2 \expe^{-r} \left( \int_0^{\expe^{-r}}\dots  \dd t + \int_{\expe^{-r}}^{\expe^r}\dots  \dd t + \int_{\expe^r}^{\pinfty} \dots  \dd t \right) \\
&\eqqcolon 2 \expe^{-r}( A_1 + A_2 + A_3),
\end{align*}
say.
Then
\[
A_1 \leq \int_0^{\expe^{-r}} (\expe^{-2r} +t^2)^{- \lambda -\sfrac{1}{2}}  \dd t \leq \expe^{2r(\lambda +\sfrac{1}{2})} \expe^{-r} = \expe^{2r\lambda},
\]
while
\[
A_2 \leq  \int_{\expe^{-r}}^{\expe^r} (\expe^{-2r} +t^2)^{- \lambda -\sfrac{1}{2}}  \dd t \leq \int_{\expe^{-r}}^{\expe^r} t^{- 2\lambda -1}  \dd t \lesssim \frac{1}{|\lambda|} \expe^{2|\lambda| r},
\]
and finally
\[
\begin{aligned}
A_3
&\leq  \int_{\expe^r}^{\pinfty} (1 +\expe^{-2r}t^2)^{ \lambda -\sfrac{1}{2}} t^{-2\lambda -1} \dd t \\
&= \expe^{-2\lambda r}\int_{1}^{\pinfty} (1 +s^2)^{ \lambda -\sfrac{1}{2}} s^{-2\lambda -1} \dd s \lesssim  \expe^{-2\lambda r},
\end{aligned}
\]
which completes the proof of the claim and of the theorem.
\end{proof}

\subsection{Optimality}
Theorem \ref{teo_principal} is best possible, in the sense that the term $(|s|+1)/|s|$ cannot be significantly improved.
Indeed, the asymptotic formulae of Bargmann \cite{Bargmann} imply that
\[
\limsup_{r \to \pinfty} \expe^{r} \labs\langle \pi_{\frac{1}{2}is, \epsilon} (a_r) f_{\frac{1}{2}is,\mu}, f_{\frac{1}{2}is,\nu} \rangle\rabs = C
\begin{cases}
|s|^{-1/2} \coth^{1/2}(\pi |s|)   & \text{when $\epsilon = 0$}, \\
|s|^{-1/2} \tanh^{1/2}(\pi |s|)   & \text{when $\epsilon = 1$}.
\end{cases}
\]
This implies immediately that we cannot do better when $|s| \leq 1$, but suggests that $(|s|+1)/|s|$ may not be optimal when $|s| \geq 1$.
However, it is evident that
\[
\labs\langle \pi_{\frac{1}{2}is, \epsilon} (a_0) f_{\frac{1}{2}is,\mu}, f_{\frac{1}{2}is,\mu} \rangle\rabs = 1,
\]
and so no bound that vanishes when $|s| \to \pinfty$ can be valid.

Howe and Tan \cite{Howe-Tan} found a uniform estimate for the principal series which also holds when $s = 0$, namely,
\[
\labs \langle \pi_{\half is, \epsilon} (a_r) f_{\half is,\mu}, f_{s,\nu} \rangle  \rabs
\lesssim (1+r)  \expe^{-r}
\]
for all $\epsilon \in \lset 0,1 \rset$, all $r\in \Rplus$ and all $\mu,\nu \in 2\Z+\epsilon$.

The proof of this estimate goes as follows; we first show that
\[
\labs \langle \pi_{\frac{1}{2}is, \epsilon} (a_r) f_{\frac{1}{2}is,\mu}, f_{\frac{1}{2}is,\nu} \rangle  \rabs
\lesssim
\expe^{-r} \int_\R (\expe^{-2r} +t^2)^{-\sfrac{1}{2}} (1+ \expe^{-2r} t^2)^{-1/2}  \dd t,
\]
by estimating as in the proof of Theorem \ref{teo_complementary},
and then break this integral into three parts and estimate each, much as we estimated $A_1$, $A_2$ and $A_3$ above.
However, the integral corresponding to $A_2$ is estimated as follows:
\[
\int_{\expe^{-r}}^{\expe^r} (\expe^{-2r} +t^2)^{ -\sfrac{1}{2}}  \dd t
\leq \int_{\expe^{-r}}^{\expe^r} t^{ -1}  \dd t
\lesssim r
\]
when $r \in [1,\pinfty)$.
This approach also gives complementary series estimates that do not blow up when the parameter $\lambda$ approaches $0$.

\section{The discrete series}
Conjugation with the matrix
\[
\begin{pmatrix}
        1 & i \\
        i & 1
      \end{pmatrix}
\]
converts $\mathrm{SL}(2,\R)$ into the group $\mathrm{SU}(1,1)$ of complex matrices of determinant $1$ that preserve the quadratic form $B$ defined by $B(w,z) \coloneqq |w|^2 - |z|^2$.
We follow Bargmann~\cite{Bargmann} (with minor notational differences), and refer in particular to~\cite[\textsection 9]{Bargmann}.
Hence, we identify the elements of $G$ with the matrices
\begin{align}\label{Gdiscrete}
x=
\begin{pmatrix}
  \alpha & \bar{\beta} \\
  \beta & \bar{\alpha}
 \end{pmatrix},
\end{align}
where $|\alpha|^2-|\beta|^2=1$.
For $\ell\in \Zplus$, we denote by $\mathcal{H}_\ell$ the Hilbert space of holomorphic functions on the unit disk $D \coloneqq \lset z\in \C: |z|<1\rset $ endowed with the inner product
\[
\langle f,g\rangle_\ell
\coloneqq \frac{\ell-1}{2\pi i} \int_D (1-|z|^2)^{\ell-2} \bar{f}(z) g(z) \dd \sigma(z)
\]
if $\ell>1$, and when $\ell = 1$,
\[
\langle f,g\rangle_1 \coloneqq \lim_{\ell \to 1+} \frac{\ell -1}{\pi} \int_D (1-|z|^2)^{\ell-2} \bar{f}(z) g(z) \dd \sigma(z),
\]
where $\sigma$ denotes Lebesgue measure in $\C$.
Write $h$ for $\half\ell$.

\subsection{The discrete series $D^+_\oldk$.}
Consider the action of $G$ on $D$ given by
\[
xz \coloneqq \frac{\bar{\alpha} z + \bar{\beta}}{\beta z + \alpha}
\qquad\forall z\in D,
\]
and the representation $\pi_\ell^+$ on $\mathcal{H}_\ell$ given by
\[
\pi_\ell^+(x) f(z) \coloneqq \mu^+(x,x^{-1}z)^\ell f(x^{-1}z)
\qquad\forall x \in D \quad\forall z \in D,
\]
where $\mu^+(x,z) \coloneqq \beta z + \alpha$.
Then $\pi_\ell^+$ is an irreducible unitary representation of $G$ on $\mathcal{H}_\ell$, and we say that it belongs to the class $D^+_\oldk$.
The functions
\[
g_{\ell, m}^+(z)\coloneqq (-1)^{m-\oldk } \binom{\ell-1+m-\oldk}{m-\oldk}^{1/2}z^{m-\oldk}
\qquad\forall m\in \oldk+\N,
\]
form an orthonormal basis of $\mathcal{H}_\ell$. Define now
\[
\oldazeta \coloneqq \begin{pmatrix}
\cosh\oldzeta & \sinh\oldzeta \\
\sinh\oldzeta & \cosh\oldzeta
\end{pmatrix}
\qquad\text{and}\qquad
\oldh_\theta \coloneqq
\begin{pmatrix}
\expe^{-i\theta} &0 \\
0 & \expe^{i\theta}
\end{pmatrix}
\]
for all $\oldzeta \in \R$ and $\theta \in [-\pi,\pi)$.
We shall prove the following theorem.
\begin{theorem}\label{teo_discrete+}
For all $\oldzeta \in \Rplus$, all $\ell \in \Z^+$, and all $m,n\in \half\ell + \N$,
\[
\labs  \langle g_{\ell,m}^+, \pi^+_\ell(\oldazeta) g_{\ell,n}^+\rangle_\ell \rabs \leq  (\cosh\oldzeta)^{-1}.
\]
\end{theorem}
\begin{proof}
In this proof, we omit the superscript $^+$ and hide the dependence on $\ell$.
Define
\[
v_{m,n}(x) \coloneqq \langle g_{\ell,m}, \pi_\ell(x) g_{\ell,n}\rangle_\ell
\qquad\forall x\in G,
\]
$y\coloneqq \sinh^2\oldzeta$ and $W_{m,n}(y) \coloneqq v_{m,n}(\oldazeta)$.
Then by~\cite[(11.2)]{Bargmann},
\begin{equation}\label{disc1}
\begin{aligned}
W_{m,n}(y)
&= \frac{(-1)^{n-\oldk}}{\Gamma(2\oldk)} \left( \frac{\Gamma(m+\oldk)\Gamma(n+\oldk)}{\Gamma(m+1-\oldk)\Gamma(n+1-\oldk)}\right)^{\sfrac{1}{2}} y^{-\oldk}\left(\frac{y}{y+1}\right)^{\sfrac{(m+n)}{2}} \\
&\qquad\times F(\oldk-m,\oldk-n,2\oldk, -\sfrac{1}{y})\\
&= \frac{(-1)^{n-\oldk}}{\Gamma(2\oldk)} \left( \frac{\Gamma(m+\oldk)\Gamma(n+\oldk)}{\Gamma(m+1-\oldk)\Gamma(n+1-\oldk)}\right)^{\sfrac{1}{2}} y^{-\oldk}\left(\frac{y}{y+1}\right)^{\sfrac{(m+n)}{2}} \\
&\qquad\times F(\oldk-n,\oldk-m,2\oldk, -\sfrac{1}{y})\\
&= \frac{(-1)^{n-\oldk}}{\Gamma(2\oldk)} \left( \frac{\Gamma(m+\oldk)\Gamma(n+\oldk)}{\Gamma(m+1-\oldk)\Gamma(n+1-\oldk)}\right)^{\sfrac{1}{2}} y^{-\oldk}\left(\frac{y}{y+1}\right)^{\sfrac{(m+n)}{2}} \\
&\qquad\times \binom{n+\oldk-1}{n-\oldk}^{-1} \left(\frac{y}{y+1}\right)^{\oldk-n}P_{n-\oldk}^{(2\oldk-1, m-n)} \lpar\frac{y-1}{y+1}\rpar\\
&=(-1)^{n-\oldk} \left( \frac{\Gamma(m+\oldk)\Gamma(n+1-\oldk)}{\Gamma(n+\oldk)\Gamma(m+1-\oldk)}\right)^{\sfrac{1}{2}}
y^{-\oldk} \left(\frac{y}{y+1}\right)^{\oldk+(m-n)/2} \\
&\qquad\times P_{n-\oldk}^{(2\oldk-1, m-n)} \lpar\frac{y-1}{y+1}\rpar,
\end{aligned}
\end{equation}
where $P_{\nu}^{(\alpha,\beta)}$ is the usual Jacobi polynomial, by the symmetry of the hypergeometric function and \cite[p.\ 170]{Bateman}.

Following Koornwinder et al.~\cite[eq.~(4.4)]{Koornwinder-et-al}, we define
\[
\mathbf{g}^{(\alpha,\beta)}_{\nu}(x)
= \left( \frac{\Gamma(\nu +1)\Gamma(\nu+\alpha+\beta +1)}{\Gamma(\nu +\alpha+1)\Gamma(\nu+\beta+1)} \right)^{\sfrac{1}{2}} \left( \frac{1-x}{2}\right)^{\sfrac{\alpha}{2}}\left( \frac{1+x}{2}\right)^{\sfrac{\beta}{2}} P_{\nu}^{(\alpha,\beta)}(x)
\]
for all $\alpha,\beta,\nu \in \N$ and all $x\in [-1,1]$.
By~\cite[eq.~(4.8)]{Koornwinder-et-al},
\begin{equation}\label{disc2}
\labs \mathbf{g}^{(\alpha,\beta)}_{\nu}(x) \rabs\leq 1.
\end{equation}
(This is an immediate consequence of relating this function to a matrix coefficient of a unitary representation of $\mathrm{SU}(2)$, an observation which goes back at least as far as  Vilenkin \cite{Vilenkin}.)
Now take $\alpha= 2\oldk-1$, $\beta = m-n$, $\nu = n-\oldk$, and
\[
x \coloneqq \frac{y-1}{y+1} \,,
\quad\text{whence}\quad
\frac{1-x}{2} = \frac{1}{y+1}
\quad\text{and}\quad \frac{1+x}{2} = \frac{y}{y+1} \,.
\]
Then $x \in (-1,1)$ when $y \in \Rplus$, so \eqref{disc1} implies the equality
\[
W_{m,n}(y) = \frac{(-1)^{n-\oldk} }{(y+1)^{1/2}} \, \mathbf{g}_{\nu}^{(\alpha,\beta)} (x) \,,
\]
and \eqref{disc2} yields the desired inequality immediately.
\end{proof}

Haagerup and Schlichtkrull \cite{HS} found sharper estimates for $\mathbf{g}_{\nu}^{(\alpha,\beta)} (x)$, but these do not seem to yield better inequalities for \emph{all} matrix coefficients.

\subsection{The discrete series $D^-_\oldk$.}
The construction is similar to that of $D^+_\oldk$.
We start from the group action of $G$ on $D$ given by
\[
x z \coloneqq \frac{\alpha z+\beta}{\bar{\beta}z +\bar{\alpha}}
\qquad\forall x \in G \quad\forall z\in D,
\]
and consider the representation $\pi_\ell^-$ on $\mathcal{H}_\ell$ given by
\[
\pi_\ell^-(x) f(z) \coloneqq \mu^-(x,x^{-1}z)^\ell f(x^{-1}z),
\]
where $\mu^-(x,z)\coloneqq \bar{\alpha} + \bar{\beta}z$.
Then $\pi_\ell^-$ is irreducible and acts unitarily on $\mathcal{H}_\ell$, and we say it belongs to the class $D^-_\oldk$ with $\oldk\coloneqq\frac{1}{2}\ell$. The functions
\[
g_{\ell, m}^-(z)\coloneqq (-1)^{-\oldk-m } \binom{\ell-1-\oldk-m}{-\oldk-m}^{1/2}z^{-\oldk-m} \coloneqq g_{\ell,-m}^+(z),
\]
where $m\in -\oldk-\N$, form an orthonormal basis of $\mathcal{H}_\ell$.
Evidently
\[
{\langle g_{\ell,m}^-, \pi_\ell^-(x) g_{\ell,n}^-\rangle_\ell }
= (-1)^{-m+n}\overline{\langle g_{\ell,-m}^+, \pi_\ell^+(x) g_{\ell,-n}^+\rangle_\ell}
\]
(see~\cite[eq.\ (10.29c)]{Bargmann}), so the next result follows from Theorem~\ref{teo_discrete+}.
\begin{theorem}\label{teo_discrete-}
For all $\oldzeta \in \Rplus$, all $\ell \in \Z^+$, and all $m,n\in -\half\ell - \N$,
\[
\labs  \langle g_{\ell,m}^-, \pi_\ell^-(\oldazeta) g_{\ell,n}^-\rangle_\ell \rabs \leq (\cosh\oldzeta)^{-1}.
\]
\end{theorem}

\section{From pointwise to integral estimates}
In this section, we prove integral versions, as in \eqref{eq:basic-estimate}, of the estimates of Theorems~\ref{teo_principal}, \ref{teo_complementary}, \ref{teo_discrete+} and~\ref{teo_discrete-}.
We begin with a general result.

\begin{proposition}\label{prop:general-vectors}
Let $\pi$ be a unitary representation of $G$ on the Hilbert space $\mathcal{H}$, $\PP$ be a subset of $\Z$, and $\lambda$ be a real number.
Suppose that $\mathcal{H}$ has an orthonormal basis of vectors $e_\mu$, where $\mu \in \PP$, such that $\pi(k_\theta) e_\mu = \expe^{i\mu\theta} e_\mu$ for all $\mu \in \PP$.
Then the following are equivalent:
\begin{itemize}
\item[1.] for all $r\in \Rplus$ and $\mu,\nu \in \PP$,
\[
\labs \langle \pi (a_r) e_{\mu}, e_{\nu} \rangle  \rabs\leq C \expe^{-\lambda r};
\]
\item[2.]  for all $r\in \Rplus$ and $f,g\in \mathcal{H}$,
\[
\left( \int_{-\pi}^{\pi} \int_{-\pi}^{\pi} \labs \langle \pi(k_{\theta_1}a_r k_{\theta_2}) f, g \rangle  \rabs^2 \dd \theta_1 \dd\theta_2 \right)^{1/2} \leq C \expe^{-\lambda r} \lnorm f \rnorm_{\mathcal{H}} \lnorm g \rnorm_{\mathcal{H}}.
\]
\end{itemize}
\end{proposition}

\begin{proof}
It is trivial that the second condition implies the first, so we need only prove the opposite implication.

Let $f,g\in \mathcal{H}$.
Then we may write
\[
f \eqqcolon \sum_{\mu \in \PP} b_\mu e_{\mu}
\qquad\text{and}\qquad
g\eqqcolon \sum_{\nu \in \PP} c_\nu e_{\nu};
\]
initially we suppose that only finitely many of the coefficients $b_\mu$ and $c_\nu$ are nonzero.
Thus
\begin{align*}
& \labs \langle \pi(k_{\theta_1}a_r k_{\theta_2}) f, g \rangle  \rabs^2 \\
 &\qquad = \langle \pi(a_r) \pi(k_{\theta_2}) f, \pi(k_{\theta_1}^{-1})g \rangle \overline{\langle \pi(a_r) \pi(k_{\theta_2}) f, \pi(k_{\theta_1}^{-1})g \rangle }\\
 & \qquad = \sum_{\mu,\mu',\nu,\nu'\in \PP} b_\mu \bar{b}_{\mu'} \bar{c}_\nu {c}_{\nu'} \expe^{i\theta_2 (\mu'-\mu)}\expe^{i\theta_1(\nu'-\nu)} \langle \pi(a_r)e_{\mu}, e_{\nu} \rangle \overline{\langle \pi(a_r)e_{\mu'}, e_{\nu'} \rangle}.
\end{align*}
Now, integrating twice,
\begin{align*}
 \int_{-\pi}^{\pi} \int_{-\pi}^{\pi} \labs \langle \pi(k_{\theta_1}a_r k_{\theta_2}) f, g \rangle  \rabs^2 \dd \theta_1 \dd\theta_2
 = \sum_{\mu,\nu\in \PP} \labs b_\mu \rabs^2 \labs c_\nu \rabs^2\labs  \langle \pi(a_r) e_{\mu}, e_{\nu} \rangle \rabs^2.
\end{align*}
A limiting argument using Fatou's Lemma proves the general case.
\end{proof}

For the principal and complementary series of representations, we use the notation of Section 2.
\begin{corollary}\label{cor_prin-comp}
For all $s\in \R\setminus \lset 0\rset $, $\epsilon \in \lset 0,1 \rset$, $r\in \Rplus$ and $f,g\in \mathcal{H}_{\frac{1}{2}is,\epsilon}$,
\[
\left( \int_{-\pi}^{\pi} \int_{-\pi}^{\pi} \labs\langle \pi_{\frac{1}{2}is, \epsilon} (k_{\theta_1} a_r k_{\theta_2}) f, g \rangle\rabs^2 \dd \theta_1 \dd\theta_2 \right)^{1/2} \lesssim \frac{|s|+1}{|s|} \expe^{-r}  \lnorm f \rnorm_{\mathcal{H}_{\frac{1}{2}is,\epsilon}}  \lnorm g \rnorm_{\mathcal{H}_{\frac{1}{2}is,\epsilon}}.
\]

For all $\lambda \in \pm (0,\frac{1}{2})$, $r\in \Rplus$ and $f,g\in \mathcal{H}_{\lambda}$,
\[
\left( \int_{-\pi}^{\pi} \int_{-\pi}^{\pi} \labs \llangle \pi_{\lambda} (k_{\theta_1}a_r k_{\theta_2}) f, g \rrangle  \rabs^2 \dd \theta_1 \dd\theta_2 \right)^{1/2} \lesssim \frac{1}{|\lambda|} \expe^{-(1-2|\lambda|)r}  \lnorm f \rnorm_{\mathcal{H}_{\lambda}}  \lnorm g  \rnorm_{\mathcal{H}_{\lambda}}.
\]
\end{corollary}

For the discrete series, we use the notation of Section 3.
Proposition \ref{prop:general-vectors} obviously also holds when we replace $a_r$ and $k_\theta$ by $\tilde a_r$ and $\tilde k_\theta$, so the next result is also immediate.

\begin{corollary}
Let $\ell \in \Zplus$.
Then for all $\oldzeta\in \Rplus$ and $f,g\in \mathcal{H}_{\ell}$,
\[
\lpar \int_{-\pi}^{\pi} \int_{-\pi}^{\pi} \labs  \langle f, \pi^\pm_\ell(\oldh_{\theta_1}\oldazeta \oldh_{\theta_2}) g \rangle_\ell \rabs^2 \dd\theta_1 \dd\theta_2 \rpar^{1/2}
\leq (\cosh\oldzeta)^{-1} \lnorm f \rnorm_{\mathcal{H}_{\ell}} \lnorm g  \rnorm_{\mathcal{H}_{\ell}}.
\]
\end{corollary}

Sharper estimates, where $(\cosh\oldzeta)^{-1}$ is replaced by $(\cosh\oldzeta)^{-1-\delta}$ for some positive $\delta$, cannot hold.
Indeed, suppose that $\pi$ is a unitary representation of $G$ and $\Phi$ is a $K$-bi-invariant $L^2(G)$-function such that
\[
\int_{-\pi}^{\pi}\int_{-\pi}^{\pi} \labs\langle \pi(k_{\theta_1} x k_{\theta_2})f, g \rangle\rabs^{2} \dd\theta_1 \dd\theta_2
\leq \Phi^2(x)
\]
for all $f, g \in \mathcal{H}_\pi$ of norm $1$.
Define the compact subset $S_R$ of $G$ by
\[
S_R \coloneqq \{ k_{\theta_1} a_r k_{\theta_2} : r \in [0,R], \ \theta_1, \theta_2 \in \R \}.
\]
Then for all $\delta$ in $\Rplus$, there exists $R \in \Rplus$ such that
\[
\int_{G \setminus S_R} \labs\langle \pi(x)f, g \rangle\rabs^{2} \dd x \leq \delta
\]
for all $f$ and $g$ in $\mathcal{H}_\pi$ of norm $1$ (where the integral is with respect to the Haar measure).
But the set of matrix coefficients is invariant under translation and so this cannot be.
Similarly, it is not possible for \emph{all} the matrix coefficients of a unitary representation $\pi$ to belong to $L^p(G)$ for some $p \in [1,2)$ (see \cite{Cowling-Annals}).

Bargmann \cite[eq.\ (12.8)]{Bargmann} observed that, if $\pi$ belongs to the discrete series class $D_\oldk^{\pm}$, then
\begin{equation}\label{eq:disc-series}
\lpar \int_G \labs \langle \pi(x) f, g \rangle \rabs^2  \dd x \rpar^{1/2}
= (2\oldk-1)^{-1} \lnorm f \rnorm_{\mathcal{H}_\pi} \lnorm g \rnorm_{\mathcal{H}_\pi}
\end{equation}
for all $f, g \in \mathcal{H}_\pi$.
It seems unlikely that our pointwise estimate can hold with an additional constant such as $(2\oldk-1)^{-1}$.

Later it was shown that equality \eqref{eq:disc-series} holds for all ``square-integrable'' representations of all locally compact groups, provided that $2\oldk-1$ is replaced by a suitable constant, called the formal degree of the representation, and that in many cases the matrix coefficients $\langle \pi(\cdot) f, g \rangle$ are integrable for a dense set of vectors $f$ and $g$.
See Warner \cite[Chapter 4]{Warner} for more details.

\section{The metaplectic representation}

Let $G$ be the group $\mathrm{Sp}(n,\R)$ of $2n\times 2n$ real matrices $S$ such that $S^T JS = J$, where
\[
J\coloneqq\begin{pmatrix}
0& I\\
-I&0
\end{pmatrix}.
\]
Let $\mathrm{U}(2n,\R) \coloneqq \mathrm{Sp}(n,\R)\cap \mathrm{O}(2n,\R)$
and recall the Cartan decomposition $G = K A^+ K$, where $K \coloneqq \mathrm{U}(2n,\R) $ and
\[
A^+\coloneqq \lset  \mathrm{diag}(\lambda_1, \dots, \lambda_n, \lambda_{1}^{-1}, \dots, \lambda_n^{-1}): \lambda_1 \geq \lambda_2 \geq \dots \geq \lambda_n\geq 1 \rset .
\]
The matrix map $\iota \colon \mathcal{M}(n,\C) \to \mathcal{M}(2n,\R)$, given by
\[
\iota (A+i B) \coloneqq
\begin{pmatrix}
  A  & -B  \\
  B  &  A
\end{pmatrix},
\]
identifies $\mathrm{U}(n,\C)$ with $\mathrm{U}(2n,\R)$.
Then $\mathbb{T}^n$ is a maximal torus of $K$, where
\[
\T^n \coloneqq \iota \lset  \begin{pmatrix}
  \expe^{i\theta_1} &\dots &0  \\
  \vdots &  \ddots & \vdots \\
  0 & \dots & \expe^{i\theta_n}
 \end{pmatrix} : \theta_1, \dots, \theta_n \in [-\pi,\pi) \rset .
\]
We normalise the Haar measures on $K$ and $\T^n$.

Denote by $\rho \colon G \to \mathcal{U}(L^2(\R^n))$ the (projective) metaplectic representation, see~\cite{FollandHAPS,Wong}, and $\langle \cdot , \cdot \rangle $ the inner product in $L^2(\R^n)$.

The aim of this section is to prove the following theorem.
\begin{theorem}\label{teo_meta}
Let $g = \mathrm{diag}(\lambda_1,\dots, \lambda_n, \lambda_1^{-1},\dots, \lambda_n^{-1})$, where $\lambda_1,\dots,\lambda_n\in [1,\pinfty)$.
Then for all $\phi_1,\phi_2 \in L^2(\R^n)$,
\[
\left(\int_{\T^n}\int_{\T^n} \labs \langle \rho(t_1 g t_2)\phi_1, \phi_2\rangle  \rabs^2 \dd t_1  \dd t_2 \right)^{1/2} \lesssim (\lambda_1 \dots \lambda_n)^{-1/2}\lnorm \phi_1 \rnorm_2 \lnorm \phi_2 \rnorm_2.
\]
\end{theorem}
Before the proof, we need some preliminaries.
We define  the cross-Wigner distribution $W(\phi_1,\phi_2)$ of $\phi_1$ and $\phi_2$ in $L^2(\R^n)$ by
\[
W(\phi_1,\phi_2)(x+iy)\coloneqq (2\pi)^{-n/2} \int_{\R^n} \expe^{-isy} \phi_1(x+\tfrac{s}{2}) \overline{\phi_2(x-\tfrac{s}{2})}  \dd s \qquad\forall x,y \in \R^n,
\]
see, e.g., \cite[formula~(3.1.2)]{Wong}. We recall the Moyal identity~\cite[Theorem~3.2]{Wong}
\[
\langle W(\phi_1, \phi_2), W(\widetilde \phi_1, \widetilde \phi_2)\rangle = \langle\phi_1, \widetilde\phi_1\rangle \langle\phi_2, \widetilde\phi_2\rangle,
\]
and the covariance property~\cite[Theorem 29.13]{Wong}
\[
W(\rho(g) \phi_1, \rho(g) \phi_2)(z) = W(\phi_1, \phi_2)(g^{-1}z)
\qquad\forall g\in G \quad\forall z\in \C^{n}.
\]
We shall write $W(\phi)$ for $W(\phi,\phi)$.

Denote by $h_n$, where $n\in \N$, the Hermite functions on $\R$.
We recall that, if $z\in \C$ and $n,k\in \N$, then (see~\cite[p.~113]{Wong})
\begin{equation}\label{hermiteformulas}
\begin{gathered}
W (h_n)(z) = 2(-1)^n (2\pi)^{-1/2} \mathcal{L}_n (2\labs z \rabs^2) \\
W(h_{n+k}, h_n)(z) = C(n,k) z^k \mathcal{L}_n^{(k)}(2\labs z \rabs^2) \\
W(h_n,h_{n+k})(z) = C(n,k) \bar{z}^k \mathcal{L}_n^{(k)}(2\labs z \rabs^2)
\end{gathered}
\end{equation}
where the Laguerre functions $\mathcal{L}_n^{(k)}$ are normalised in $L^2(\Rplus)$ and $\mathcal{L}_n \coloneqq \mathcal{L}_n^{(0)}$.
Therefore
\begin{equation}\label{ort_wigner}
\int_{-\pi}^{\pi} W(h_n,h_m)(\expe^{i\theta}z)  \dd \theta =
\begin{cases}
0  & \text{if $m \neq n$}, \\
2\pi W(h_n)(z)  &\text{if $m = n$}.
\end{cases}
\end{equation}
Denote now by $h_\beta$, where $\beta\coloneqq (\beta_1,\dots, \beta_n) \in \N^n$, the tensor product of one-dimensional Hermite functions, that is,
\[
h_\beta(x) \coloneqq \prod_{\ell\coloneqq1}^n h_{\beta_\ell}(x_\ell)
\qquad\forall x \coloneqq (x_1,\dots,x_n) \in \R^{n}.
\]
Since the cross-Wigner distribution is sesquilinear and compatible with tensor products,
\[
W(h_\beta,h_\beta)(z) = \prod_{\ell\coloneqq1}^n W(h_{\beta_\ell}, h_{\gamma_\ell})(z_\ell)
\qquad\forall \beta,\gamma \in \N^n
\]
(where $z \coloneqq (z_1,\dots, z_n) \in \C^n$), and in particular
\begin{equation}\label{wignertensprod}
W(h_\beta)(z) =  \prod_{\ell\coloneqq1}^n W(h_{\beta_\ell})(z_\ell).
\end{equation}
Thus, by~\eqref{ort_wigner},
\begin{equation}\label{ort_tensorwigner}
\int_{\T^n} W(h_\beta, h_\gamma)(kz) \dd k =
\begin{cases}
0  & \text{if $\beta\neq \gamma$}, \\
W(h_\beta)(z)  &\text{if $\beta = \gamma$}.
\end{cases}
\end{equation}

The idea of the proof of Theorem~\ref{teo_meta} is to reduce the estimate to the case where $n = 1$.
The following lemma will be the key to treating this case.
\begin{lemma}\label{lemmameta}
Suppose that $\phi_j\in L^2(\R^n)$ and $F_j(z) \coloneqq \int_{-\pi}^{\pi} W\phi_j (\expe^{i\theta }z)  \dd \theta$, where $j = 1,2$, and that $\lambda \in [1, \pinfty)$.
Then
\[
\iint_{\R^2} F_1(\lambda^{-1}x + iy) F_2(x + i\lambda^{-1}y)  \dd x  \dd y \lesssim \lnorm \phi_1 \rnorm_2^2 \lnorm \phi_2 \rnorm_2^2.
\]
\end{lemma}
Assuming the lemma, we prove the theorem.
\begin{proof}[Proof of Theorem~\ref{teo_meta}]
By the Moyal identity and the covariance property,
\begin{align*}
\labs \langle \rho(k_1 g k_2)\phi_1, \phi_2\rangle \rabs^2
& = \langle W(\rho(k_1gk_2)\phi_1), W(\phi_2)\rangle \\
& = \int_{\C^{n}} W(\phi_1)(k_2^{-1} g^{-1} k_1^{-1}z) W(\phi_2)(z) \dd \sigma(z)\\
& = \int_{\C^{n}} W(\phi_1)(k_2^{-1} g^{-1} z) W(\phi_2)(k_1 z) \dd \sigma(z).
\end{align*}
Thus
\begin{equation}\label{innprod_wigner}
\int_{\T^n}\int_{\T^n}\labs \langle \rho(t_1 g t_2)\phi_1, \phi_2\rangle \rabs^2   \dd t_1  \dd t_2 = \int_{\C^n} F_1(g^{-1}z) F_2(z) \dd \sigma(z),
\end{equation}
where
\[
F_j(z) \coloneqq \int_{\T^n} W(\phi_j)(kz) \dd k.
\]
We now decompose the functions $\phi_j$ as sums of Hermite functions $h_\beta$, where $\beta \in \N^n$. By~\eqref{ort_tensorwigner}, if $\phi_1 \eqqcolon \sum_\beta b_\beta h_\beta$ and  $\phi_2 \eqqcolon \sum_\gamma c_\gamma h_\gamma$, then
\begin{equation}\label{decompositionphi}
\int_{\T^n}W(\phi_1)(kz) \dd k = \sum_{\beta} \labs b_\beta \rabs^2 W(h_\beta)(z),
\end{equation}
and an analogous result holds for $\phi_2$.
By~\eqref{innprod_wigner} and~\eqref{wignertensprod}, then,
\begin{align*}
&\int_{\T^n}\int_{\T^n}\labs \langle \rho(t_1 g t_2)\phi_1, \phi_2\rangle \rabs^2   \dd t_1  \dd t_2 \\
&\quad = \sum_{\beta,\gamma}\labs b_\beta \rabs^2 \labs c_\gamma \rabs^2   \int_{\C^{n}} W(h_\beta) (g^{-1}z)  W(h_\gamma) (z)  \dd \sigma(z)\\
&\quad = \sum_{\beta,\gamma}\labs b_\beta \rabs^2 \labs c_\gamma \rabs^2  \prod_{\ell = 1}^n  \iint_{\R^2} W(h_{\beta_\ell}) (\lambda_\ell^{-1}x_{\ell} +i \lambda_\ell y_{\ell})  
W(h_{\gamma_\ell})(x_{\ell} + iy_{\ell})  \dd x_{\ell} \dd y_{\ell},
\end{align*}
where $z\coloneqq (z_1,\dots, z_n)$ and $z_\ell\coloneqq x_{\ell} + i y_{\ell}$.
Thus Theorem~\ref{teo_meta} boils down to the estimate
\[
\iint_{\R^2} W(h_{\beta_\ell}) (\lambda_\ell^{-1}x_{\ell} +i \lambda_\ell y_{\ell})  W(h_{\gamma_\ell})(x_{\ell} + i y_{\ell})  \dd x_{\ell} \dd y_{\ell} \lesssim \lambda_{\ell}^{-1}
\]
for all $\ell \in \{1,\dots, n\}$,
which follows from Lemma~\ref{lemmameta}.
\end{proof}

It remains to prove Lemma~\ref{lemmameta}.

\begin{proof}[Proof of Lemma~\ref{lemmameta}]
When $\lambda$ belongs to a bounded subset of $[1,\pinfty)$, the desired bound follows at once from the Cauchy--Schwarz inequality and the Minkowski inequality for integrals, since $\lnorm W(\phi_j) \rnorm_2 = \lnorm \phi_j \rnorm_2^2$ from the Moyal identity.
Hence we may suppose that $\lambda$ is large.

We begin by decomposing $\phi_1$ and $\phi_2$ in terms of Hermite functions $h_n$. We write
\[
\phi_1 \eqqcolon \sum_{m} b_m h_m
\qquad\text{and}\qquad
\phi_2 \eqqcolon \sum_{n} c_n h_n.
\]
By~\eqref{decompositionphi} and~\eqref{hermiteformulas}, it will be necessary and sufficient to prove that
\[
\labs  \int_{\R^2} \mathcal{L}_m (2(\lambda^{-2}x^2 + \xi^2)) \mathcal{L}_n(2(x^2 + \lambda^{-2}\xi^2)) \dd x  \dd \xi \rabs \lesssim 1
\]
for all sufficiently large $\lambda$ and $m,n\in \N$, or equivalently
\begin{equation}\label{intr1r2}
\labs  \int_0^{\pinfty} \int_0^{\pinfty} \mathcal{L}_m (\lambda^{-2}r_1 + r_2) \mathcal{L}_n(r_1 + \lambda^{-2}r_2) r_1^{-1/2}r_2^{-1/2} \dd r_1  \dd r_2 \rabs \lesssim 1,
\end{equation}
which, by the change of variables $\lambda^{-2}r_1+r_2 \eqqcolon u$, $r_1 + \lambda^{-2}r_2 \eqqcolon v$, is in turn equivalent to
\begin{align*}
\labs  \iint_{S_{\epsilon}} \mathcal{L}_m (u) \mathcal{L}_n(v) (v-\epsilon u)^{-1/2}(u-\epsilon v)^{-1/2} \dd u  \dd v \rabs \lesssim 1
\end{align*}
for all small positive $\epsilon$, where $S_\epsilon \coloneqq\lset  (u,v)\in \Rplus \times \Rplus : \epsilon v\leq u \leq \epsilon^{-1}v\rset $.
We will in fact show that
\begin{align}\label{toprove1}
\iint_{S_{\epsilon}} \labs \mathcal{L}_m(u) \rabs \labs \mathcal{L}_n(v) \rabs \labs v-\epsilon u \rabs^{-1/2} \labs u-\epsilon v \rabs^{-1/2}  \dd u \dd v \lesssim 1
\end{align}
for all $\epsilon \in (0, \half]$.

The first step is an estimate for one integral.
We know that
\[
\labs \mathcal{L}_n(x) \rabs\lesssim A_n(x) +B_n(x) +C(x)
\qquad\forall x \in \Rplus,
\]
where, by~\cite[pp.\ 27--28]{Thangavelu},
\begin{align}
A_n(x) &\coloneqq (\nu x)^{-1/4} \mathbf{1}_{[0,\nu/2]}(x), \label{estAn}\\
B_n(x) &\coloneqq \nu^{-1/4}\labs \nu-x \rabs^{-1/4} \mathbf{1}_{[\nu/2, 3\nu/2]}(x)\label{estBn},\\
C(x)&\coloneqq \expe^{-\tau x};\label{estC}
\end{align}
here $\nu \coloneqq 4n+2$ and $\tau$ is a suitable (positive) constant.
Hence
\begin{equation}\label{eq:Laplace-est}
\int_0^{\pinfty} \labs\mathcal{L}_n(x)\rabs \labs x - s\rabs^{-1/2}  \dd x \lesssim 1
\end{equation}
for all $n \in \N$ and $s \in \R$.
Indeed, set $s \coloneqq \nu t$ and $x \coloneqq \nu y$; then
\[
\int_0^{\nu/2}  (\nu x)^{-1/4} \labs x - s \rabs^{-1/2}  \dd x
=  \int_0^{1/2}  y^{-1/4} \labs y - t \rabs^{-1/2}  \dd y
\lesssim 1
\]
uniformly in $t$; moreover, similarly,
\[
\int_{\nu/2}^{3\nu/2}  \nu^{-1/4} \labs \nu - x \rabs^{-1/4} \labs x - s \rabs^{-1/2}  \dd x
= \int_{1/2}^{3/2} \labs 1 - x \rabs^{-1/4} \labs x - t \rabs^{-1/2}  \dd x
\lesssim 1
\]
uniformly in $t$; and finally,
\[
\int_0 ^{\pinfty} \expe^{-\tau x} \labs x - s \rabs^{-1/2}  \dd x
\lesssim 1
\]
uniformly in $s$.

By symmetry, it will suffice to estimate the integral \eqref{toprove1} over the region $T_\epsilon \coloneqq \lset  (u,v) : \epsilon v \leq u \leq v  \rset$.
In this region, $v - \epsilon u \geq v/2$ since $\epsilon \in (0, \half]$, and so
\begin{align*}
&\iint_{T_{\epsilon}} \labs \mathcal{L}_m (u) \rabs \labs \mathcal{L}_n(v) \rabs \labs v-\epsilon u\rabs^{-1/2} \labs u-\epsilon v\rabs^{-1/2} \dd u \dd v \\
&\qquad\lesssim \iint_{T_{\epsilon}} \labs \mathcal{L}_m (u) \rabs \labs \mathcal{L}_n(v) \rabs \labs v\rabs ^{-1/2} \labs u-\epsilon v \rabs^{-1/2} \dd u  \dd v \\
&\qquad\leq \int_{0}^{\pinfty} \int_{0}^{\pinfty} \labs \mathcal{L}_m (u) \rabs \labs \mathcal{L}_n(v) \rabs \labs v\rabs^{-1/2} \labs u-\epsilon v\rabs^{-1/2}  \dd u  \dd v \\
&\qquad\leq \int_{0}^{\pinfty} \labs \mathcal{L}_n(v) \rabs \labs v\rabs^{-1/2}  \dd v
\lesssim 1,
\end{align*}
by two applications of \eqref{eq:Laplace-est}, as required.
\end{proof}

Before stating a corollary, we recall that $K\coloneqq \mathrm{U}(2n, \R)$.

\begin{corollary}\label{corsingval}
Let $g\in G$ and let $\lambda_1,\dots, \lambda_n,\lambda_n^{-1},\dots, \lambda_1^{-1}$ be its singular values, arranged so that $\lambda_1,\dots, \lambda_n \in [1, \pinfty)$.
Then for all $\phi_1,\phi_2 \in L^2(\R^n)$,
\[
\left(\int_{K}\int_{K} \labs \langle \rho(k_2gk_1)\phi_1, \phi_2\rangle  \rabs^2 \dd k_1  \dd k_2 \right)^{1/2} \lesssim (\lambda_1 \dots \lambda_n)^{-1/2}\lnorm \phi_1 \rnorm_2 \lnorm \phi_2 \rnorm_2.
\]
\end{corollary}
\begin{proof}
By the Cartan decomposition of $g$ and a change of variables, we see that it suffices to consider the case where $g = \mathrm{diag}(\lambda_1,\dots, \lambda_n, \lambda_1^{-1},\dots, \lambda_n^{-1})$ and $\lambda_1 \geq \dots \geq \lambda_n \geq 1$.  Then
\begin{align*}
& \int_{K}\int_{K} \labs \langle \rho(k_2gk_1)\phi_1, \phi_2\rangle  \rabs^2 \dd k_1  \dd k_2 \\
&\qquad = \int_{K} \int_{K} \int_{\T^n}\int_{\T^n}  \labs \langle \rho({k}_2 t_2g t_1 {k}_1)\phi_1, \phi_2\rangle  \rabs^2  \dd t_1  \dd t_2  \dd {k}_1 \dd {k}_2,
\end{align*}
as the Haar measure on $K$ is invariant under translations on both sides.
Now
\[
\langle \rho({k}_2 t_2g t_1 {k}_1)\phi_1, \phi_2\rangle
= \langle \rho(t_2g t_1)\rho({k}_1)\phi_1, \rho({k}_2^{-1})\phi_2\rangle,
\]
so that the statement follows by applying Theorem~\ref{teo_meta} to the functions $\rho({k}_1)\phi_1$ and $\rho({k}_2^{-1})\phi_2$, which have the same norm as $\phi_1$ and $\phi_2$ since $\rho$ is unitary.
\end{proof}

Now we present, as a consequence of the above results, a new fixed-time estimate of dispersive type for the Schr\"odinger equation in $\R^n$. In fact the propagator of the free Schr\"odinger equation is a particular metaplectic operator (see, e.g., \cite{FollandHAPS, Wong}).

\begin{theorem}\label{Theorem-5.4}
Let $t\in \R$. Then for all $\phi_1,\phi_2 \in L^2(\R^n)$
\begin{equation*}
\left(\int_K\int_K \labs \langle \expe^{it
\Delta}\rho(k_1)\phi_1,\rho(k_2) \phi_2\rangle\rabs ^2 \dd k_1 \dd k_2 \right)^{1/2}\lesssim (1+|t|)^{-n/2}\lnorm \phi_1 \rnorm_2\lnorm \phi_2 \rnorm_2.
\end{equation*}
\end{theorem}
\begin{proof}
Up to a constant of modulus $1$, $\expe^{i{t} \Delta}\coloneqq\rho(g_t)$, where
  \[
  \mathrm{Sp}(n,\R)\ni g_t\coloneqq
  \begin{pmatrix}
  I_n & 2t I_n\\
  0& I_n
  \end{pmatrix},
  \]
where $I_n$ is the $n\times n$ identity matrix, see~\cite[Proposition 29.10]{Wong}.
Hence we may apply Corollary~\ref{corsingval}.
To compute the singular values of $g_t$, we observe that $g_t$ is similar to a direct sum $c_t \oplus \dots \oplus c_t$ (with $n$ summands), where
\[
c_t
\coloneqq
\begin{pmatrix}
  1 & 2t \\
  0 & 1
\end{pmatrix}
\qquad\text{whence}\qquad
c_t^Tc_t
\coloneqq
\begin{pmatrix}
  1 & 2t \\
  2t& 1+4t^2
\end{pmatrix}.
\]
The eigenvalues of $c_t^Tc_t$ are $1+2t^2 \pm 2(t^4+t^2)^{1/2}$.
Hence the $n$ singular values of $g_t$ that are at least $1$ are all equal to $(1+2t^2 + 2(t^4+t^2)^{1/2})^{1/2}$, and the desired estimate follows.
\end{proof}

\begin{remark}
The estimate of Theorem \ref{Theorem-5.4} is similar to the usual dispersive estimate for $ \expe^{i{t} \Delta}$ from $L^1\to L^\infty$, but averaging on $K$ gives an $L^2$ estimate.\par
Note that, when the dimension $n$ is $1$, the group $K$ is just the circle group $\T^1$ and for $k_\theta\in \T^1$, the metaplectic operator $\rho(k_\theta)$ is just the fractional Fourier transform, where $\theta\in[-\pi,\pi)$.
\end{remark}

\newcommand{\arttitle}[1]{\lq #1\rq}
\newcommand{\journal}[1]{\textit{#1}}
\newcommand{\jvol}[1]{\textbf{#1}}
\newcommand{\booktitle}[1]{\textit{#1}}


\begin{thebibliography}{10}

\bibitem{ACD1}
F. Astengo, M.G. Cowling, B. Di Blasio,
\arttitle{Uniformly bounded representations and completely bounded multipliers of $\mathrm{SL}(2,\R)$},
\journal{Adv.\ Oper.\ Theory} \jvol{3} (2018), no.\ 1, 247--270.

\bibitem{Bargmann}
V.\ Bargmann,
\arttitle{Irreducible unitary representations of the Lorentz group},
\journal{Ann.\ of Math. (2)} \textbf{48} (1947), 568--640.

\bibitem{Bateman}
 A.\ Erd\'elyi, W.\ Magnus, F.  Oberhettinger, F. Tricomi,
\booktitle{Higher Transcendental Functions. Vol. II. Based on Notes Left by Harry Bateman}.
Robert E. Krieger Publishing Co., Inc., Melbourne, Fla., 1981.

\bibitem{Benoist-Kobayashi}
Y.~Benoist, T.~Kobayashi,
\arttitle{Tempered reductive homogeneous spaces},
\journal{J. Eur. Math. Soc. (JEMS)} \jvol{17} (2015), no. 12, 3015--3036.

\bibitem{Casselman-Milicic}
W. Casselman, D. Mili\v{c}i\'{c},
\arttitle{Asymptotic behavior of matrix coefficients of admissible representations},
\journal{Duke Math. J.} \jvol{49} (1982), 869--930.

\bibitem{CNT}
A.\ Cauli, F.\ Nicola, A.\ Tabacco,
\arttitle{Strichartz estimates for the metaplectic representation},
\journal{Rev.\ Mat.\ Iberoam.}, to appear. arXiv:1706.03615

\bibitem{Cowling-Annals}
M. Cowling,
\arttitle{The Kunze--Stein phenomenon},
\journal{Ann.\ of Math. (2)} \textbf{107} (1978), 209--234.

\bibitem{Cowling-Nancy}
M. Cowling,
\arttitle{Sur les coefficients des repr\'{e}sentations unitaires des groups de Lie semisimples},
in:
\booktitle{Analyse harmonique sur les groupes de Lie II. S\'eminaire Nancy-Strasbourg 1976--78},
edited by P.~Eymard, R.~Takahashi, J.~Faraut, G.~Schiffmann.
\textit{Lecture Notes in Math.} \textbf{739}.
Springer, Berlin--Heidelberg--New York, 1979, 132--178.

\bibitem{CMS}
M. Cowling, S. Meda, A.G. Setti,
\arttitle{An overview of harmonic analysis on the group of isometries of a homogeneous tree},
\journal{Exposition. Math.} \jvol{16} (1998), no. 5, 385--423.

\bibitem{Ehrenpreis-Mautner}
L.~Ehrenpreis, F.~Mautner,
\arttitle{Uniformly bounded representations of groups}
\journal{Proc. Nat. Acad. Sci. USA} \jvol{41} (1955), 231--233.

\bibitem{FollandHAPS}
G.B.~Folland,
\booktitle{Harmonic Analysis in Phase Space}.
Annals of Mathematics Studies, 122.
Princeton University Press, Princeton, NJ, 1989.

\bibitem{Haagerup-free}
U.~Haagerup,
\arttitle{An example of a nonnuclear $C^*$-algebra, which has the metric approximation property},
\journal{Invent. Math.} \jvol{50} (1978/79), no. 3, 279--293.

\bibitem{HS}
U.~Haagerup, H.~Schlichtkrull,
\arttitle{Inequalities for Jacobi polynomials},
\journal{Ramanujan J.} \jvol{33} (2014), no.\ 2, 227--246.

\bibitem{Howe-CIME}
R.~Howe,
\arttitle{On a notion of rank for unitary representations of the classical groups},
in: \booktitle{Harmonic Analysis and Group Representations}.
Liguori, Naples, 1982, 223--331.  


\bibitem{Howe-Tan}
R.~Howe, E.-C.~Tan,
\booktitle{Nonabelian Harmonic Analysis. Applications of $\mathrm{SL}(2,\R)$}.
Universitext.
Springer-Verlag, New York, 1992.

\bibitem{knapp}
A.W.~Knapp,
\booktitle{Representation Theory of Semisimple Groups.
An Overview Based on Examples.}
Princeton Mathematical Series, 36. Princeton University Press, Princeton, NJ, 1986.

\bibitem{Koornwinder-et-al}
T.~Koornwindera, A.~Kostenko and G.~Teschl,
\arttitle{Jacobi polynomials, Bernstein-type inequalities and dispersion estimates for the discrete Laguerre operator},
\journal{Adv. Math.} \jvol{333} (2018), 796--821.

\bibitem{Kunze-Stein}
R.A.~Kunze, E.M.~Stein,
\arttitle{Uniformly bounded representations and harmonic analysis on the $2\times 2$ real unimodular group},
\journal{Amer. J. Math.}
\jvol{82} (1960), 1--62.

\bibitem{Nevo}
A.~Nevo,
\arttitle{Exponential volume growth, maximal functions on symmetric spaces, and ergodic theorems for semi-simple Lie groups},
\journal{Ergodic Theory Dynam. Systems} \jvol{25} (2005), no. 4, 1257--1294.

\bibitem{tao}
T.\ Tao,
\booktitle{Nonlinear Dispersive Equations}.
CBMS 106.
Amer. Math. Soc.,  2006.

\bibitem{Thangavelu}
S.\ Thangavelu,
\booktitle{Lectures on Hermite and Laguerre Expansions}.
Mathematical Notes, 42.
Princeton University Press, Princeton, NJ, 1993.

\bibitem{Vilenkin}
N.Ya.\ Vilenkin,
\booktitle{Special Functions and the Theory of Group Representations}.
Translations of Mathematical Monographs Vol.~22.
American Mathematical Society, Providence, 1968.

\bibitem{wallach}
N.R.~ Wallach,
\booktitle{Real Reductive Groups. I.}
Pure and Applied Mathematics, 132.
Academic Press, Inc., Boston, MA, 1988.

\bibitem{Warner}
G. Warner,
\booktitle{Harmonic Analysis on Semi-simple Lie Groups. II}.
Grundlehren der mathematischer Wissenschaft, Band 189.
Springer-Verlag, Heidelberg--New York, 1972.


\bibitem{Wong}
M.W.~Wong,
\booktitle{Weyl Transforms}.
Universitext.
Springer-Verlag, New York, 1998.
\end{thebibliography}
\end{document}